\documentclass[onecolumn,11pt,aps,nofootinbib,superscriptaddress,tightenlines,showkeys]{revtex4}

\usepackage[T1]{fontenc}
\usepackage[utf8]{inputenc}

\usepackage[cmex10]{amsmath}
\usepackage{graphics}
\usepackage{dsfont}
\usepackage{amssymb}
\usepackage{graphicx}
\usepackage{mathrsfs}
\usepackage{units}
\usepackage{pgfplots}
\usepackage{enumitem}
\usepackage[colorlinks=true,linkcolor=black,citecolor=black,plainpages=false,pdfpagelabels]{hyperref}
\hypersetup{pdftitle={Eigenvalue estimates for the resolvent of a non-normal matrix}}

\usepackage{amsthm}
\newtheorem{theorem}{Theorem}[section]
\newtheorem{lemma}[theorem]{Lemma}
\newtheorem{proposition}[theorem]{Proposition}
\newtheorem{corollary}[theorem]{Corollary}

\newcommand{\abs}[1]{\ensuremath{|#1|}}

\newcommand{\Abs}[1]{\ensuremath{\left|#1\right|}}

\newcommand{\norm}[2]{\ensuremath{|\!|#1|\!|_{#2}}}

\newcommand{\Norm}[2]{\ensuremath{\left|\!\left|#1\right|\!\right|_{#2}}}

\newcommand{\tr}{\textnormal{tr}}

\newcommand{\braket}[2]{\langle #1 | #2 \rangle}

\renewcommand{\d}[1]{\ensuremath{\textnormal{d}#1}}

\newcommand{\id}{\ensuremath{\mathds{1}}}

\newcommand{\cA}{\mathcal{A}}

\newcommand{\cC}{\mathcal{C}}

\newcommand{\cM}{\mathcal{M}}

\newcommand{\cO}{\mathcal{O}}

\begin{document}
\title{Eigenvalue estimates for the resolvent of a non-normal matrix}
\author{Oleg Szehr}
\email{oleg.szehr@posteo.de}
\affiliation{Department of Mathematics, Technische Universit\"{a}t M\"{u}nchen, 85748 Garching, Germany}

\date{\today}

\begin{abstract} We investigate the relation between the spectrum of a non-normal matrix and the norm of its resolvent. We provide spectral estimates for the resolvent of matrices whose largest singular value is bounded by $1$ (so-called Hilbert space contractions) and for power-bounded matrices. In the first case our estimate is optimal and we present explicit matrices that achieve equality in the bound. This result recovers and generalizes previous estimates obtained by E.B.~Davies and B.~Simon in the study of orthogonal polynomials on the unit circle. In case of power-bounded matrices we achieve the strongest estimate so far. Our result unifies previous approaches, where the resolvent was estimated in certain restricted regions of the complex plane.
To achieve our estimates we relate the problem of bounding the norm of a function of a matrix to a Nevanlinna-Pick interpolation problem in a corresponding function space. In case of Hilbert space contractions this problem is connected to the theory of compressed shift operators to which we contribute by providing explicit matrix representations for such operators. Finally, we apply our results to study the sensitivity of the stationary states of a classical or quantum Markov chain with respect to perturbations of the transition matrix.
\end{abstract}

\maketitle

\tableofcontents
\section{Introduction}
The contribution of this article is to provide new estimates on the norm of the resolvent of a matrix $A$ and to prove their optimality under certain conditions. We derive bounds of the form
\begin{align}
\Norm{(\zeta-A)^{-1}}{}\leq \Phi(\zeta,n,\sigma(A)),\label{todo}
\end{align}
where $\Phi$ is a function of $\zeta\in\mathbb{C}$, the dimension $n$ and the spectrum $\sigma(A)$ of $A$.
In the first part of the article, (cf.~Section~\ref{contractions}) we assume that the largest singular value (the spectral norm) of $A$ is bounded by $1$ i.e.~$\Norm{A}{\infty}\leq 1$. Note that this can always be achieved by a suitable normalization. Under this assumption we obtain optimal bounds for $\zeta\in\mathbb{C}-\sigma(A)$ and present explicit matrices that establish equality in \eqref{todo}.
Thus we identify the relation between the localization of the spectrum of $A$ and the norm of its resolvent. 
In the second part (cf.~Section~\ref{powerbounds}) we study~\eqref{todo} under the assumption that each power of $A$ can be bounded with respect to \emph{any} given norm by the same constant, $\sup_{k\geq0}\Norm{A^k}{}\leq C$. In this case we derive the strongest estimates so far.

The problem of finding good functions $\Phi$ in \eqref{todo} was studied  in the literature before~\cite{BS,Rachid1,AC, LOTS}. Our approach is based on the theory of certain (Hilbert/ Banach) function spaces. We associate to a given class of matrices $\Gamma$ a certain Banach algebra $\cA$ of functions and instead of working with matrices directly we estimate the norm of a representative function in the function algebra. A key role is played by inequalities of the type
\begin{align}
\Norm{f(A)}{}\leq C\Norm{f}{\cA}\label{fcalc},
\end{align}
which relate for a given $A\in\Gamma$ the norm $\Norm{f(A)}{}$ to the norm of $f$ in $\cA$.
At first glance this appears to be of little use since the right hand side no longer depends on $A$. However, it is possible to exploit spectral properties of $A$ to significantly strengthen the inequality in~\eqref{fcalc}. Let $m_A$ be the minimal polynomial of $A$.  For any $f,g\in\cA$ we have then that $\Norm{(f+m_Ag)(A)}{\cA}=\Norm{f(A)}{\cA}$ and an application of \eqref{fcalc} reveals that for all $g\in\cA$ we have $\Norm{f(A)}{}\leq C\Norm{f+m_Ag}{\cA}$. This relates the problem of bounding $\Norm{f(A)}{}$ to the problem of finding the least norm function $f+m_Ag$ in $\cA$, which is equivalent to a Nevanlinna-Pick interpolation problem in $\cA$~\cite{AC}. If $\Norm{A}{\infty}\leq1$ the resulting interpolation problem can be solved using an operator theoretic approach pioneered by D.~Sarason~\cite{Sarason,interpol}. This approach is intrinsically connected to the theory of compressed shift operators on Hardy space. We contribute to this theory by providing a framework that allows us to compute explicit matrix representations for functions of model operators.
In case that $A$ is power-bounded we choose a rational approximation function in $\cA$ and bound its norm to achieve our result.

Bounds on the norm of a resolvent occur in various situations in pure and applied mathematics. For example in operator theory, when constructing a functional calculus~\cite{AC}
$$f\mapsto f(X)=\frac{1}{2\pi i}\int_\gamma f(\zeta)(\zeta-X)^{-1}\d\zeta.$$
In the theory of orthogonal polynomials, when studying the location of zeros of random orthogonal polynomials~\cite{BS}. In computational linear algebra the following are classical problems that can be approached through appropriate estimates for $\Norm{(\zeta-A)^{-1}}{}$.
\begin{enumerate}
\item To analyze the stability of solutions $x$ of the matrix equation $Ax-\zeta x=b$ under perturbations in $b$ and $A$, see \cite{LOTS}.
\item To study whether an approximate eigenvalue $\zeta$ of $A$ (in the sense that $\Norm{Ax-\zeta x}{}\leq\varepsilon\norm{x}{}$ for some vector $x\neq0$) is close to an actual eigenvalue of $A$, see \cite{wielandt,henrici,LOTS}.
\item To estimate the distance of the spectrum of a matrix $B$ to the spectrum of a matrix $A$ in terms of $B-A$, see \cite{phillips,henrici,bandtlow}.
\end{enumerate}
Our resolvent bounds are stronger than the ones used for example in~\cite{phillips} to obtain estimates on the spectral variation of non-normal matrices. 
In Section~\ref{Chain} we apply our estimate for power-bounded matrices to study the sensitivity of stationary states of a classical or quantum Markov chain under perturbations of the transition matrix. We recover known stability results for classical Markov chains and prove new estimates in the quantum case. A similar approach, based on the power-boundedness of the transition matrix, was previously applied in \cite{szrewo} to investigate spectral convergence properties of classical and quantum Markov chains.

\section{Preliminaries}\label{PREL}
We will take a function space based approach to the problem of bounding the norm of the resolvent of a certain matrix. This section lays down the required definitions and basic results.
\subsection{Notation}
We denote by $\cM_n$ the set of $n\times n$ matrices with complex entries. For $A\in\cM_n$ we denote by $\sigma(A)$ its spectrum and by $m$ its minimal polynomial. We write $\abs{m}$ for the degree of $m$. To the minimal polynomial $m$ we associate the Blaschke product
\begin{align*}
B(z):=\prod_{i}\frac{z-\lambda_i}{1-\bar{\lambda}_iz}.
\end{align*}
The product is taken over all $i$ such that (respecting multiplicities) the corresponding linear factor $z-\lambda_i$ occurs in the minimal polynomial $m$. Thus, the numerator of $B$ as defined here is exactly the associated minimal polynomial.

We denote by $\Norm{A}{}$ \emph{any} particular norm of $A$ while the $\infty$-norm is defined by
\begin{align*}
\Norm{A}{\infty}=\sup_{\norm{v}{2}=1}\Norm{Av}{2},
\end{align*}
where $\Norm{v}{2}^2=\sum_i\abs{v_i}^2$ is the usual Euclidean norm. That means $\Norm{A}{\infty}$ simply denotes the largest singular value of $A$. We will slightly abuse nomenclature and call matrices with 
\begin{align*}
\Norm{A}{\infty}\leq1
\end{align*}
Hilbert space contractions, although of course the underlying space always has finite dimension. Similarly, the class of $A\in\cM_n$ 
with
\begin{align*}
\sup_{k\geq0}\norm{A^k}{}\leq C<\infty
\end{align*}
will be called Banach space power-bounded operators with respect to $\Norm{\cdot}{}$ and constant $C$. (Note that here the norm is general.)

Let $\mathbb{D}=\{z\in\mathbb{C}\:|\:\abs{z}<1\}$ denote the open unit disk in the complex plane and $\bar{\mathbb{D}}$ its closure. The space of analytic functions on $\mathbb{D}$ is denoted by $Hol(\mathbb{D})$. The Hardy spaces considered here are
\begin{align*}
H_2:=\big\{f\in Hol(\mathbb{D})| \Norm{f}{H_2}^2:=\sup_{0\leq r<1}\frac{1}{2\pi}\int_0^{2\pi}\abs{f(re^{i\phi})}^2\d{\phi}<\infty\big\},
\end{align*}
and 
\begin{align*}
&H_\infty:=\big\{f\in Hol(\mathbb{D})| \Norm{f}{H_\infty}:=\sup_{z\in\mathbb{D}}\abs{f(z)}<\infty\big\}.
\end{align*}
The $H_2$-norm can be written in terms of the Taylor coefficients of the analytic function $f$. We write $f(z)=\sum_{k\geq0}\hat{f}(k)z^k$ and use Plancherel\rq{}s identity to conclude that
\begin{align*}
\sup_{0\leq r<1}\frac{1}{2\pi}\int_0^{2\pi}\abs{f(re^{i\phi})}^2\d{\phi}=\sum_{k\geq0}\abs{\hat{f}(k)}^2.
\end{align*}
Thus, $f\in Hol(\mathbb{D})$ is in $H_2$ if and only if $\sum_{k\geq0}\abs{\hat{f}(k)}^2<\infty$. The Wiener algebra is defined as the subset of $Hol(\mathbb{D})$ of absolutely convergent Taylor series,
\begin{align*}
W:=\{f=\sum_{k\geq0}\hat{f}(k)z^k| \Norm{f}{W}:=\sum_{k\geq0}\abs{\hat{f}(k)}<\infty\}.
\end{align*}
\subsection{Model spaces and operators}
Let $A\in\cM_n$ with $\sigma(A)\subset\mathbb{D}$ and let $B$ be the Blaschke product associated to the minimal polynomial of $A$.
We define the $\abs{m}$-dimensional \emph{model space}
\begin{align*}
K_B:=H_2\ominus BH_2:=H_2\cap(BH_2)^\bot,
\end{align*}
where we employ the usual scalar product from the Hilbert space $L^2(\partial\mathbb{D})$,
\begin{align*}
\braket{f}{g}:=\int_{\partial\mathbb{D}}f(z)\overline{g(z)}\:\frac{\abs{\d z}}{2\pi}.
\end{align*}

If the zeros $\{\lambda_i\}_{i=1,...,\abs{m}}$ of $B$ are distinct (that is $A$ can be diagonalized) it is not difficult to verify that $K_B$ is spanned by the \emph{Cauchy kernels}
\begin{align*}
K_B=span\left\{\frac{1}{1-\bar{\lambda}_iz}\right\}_{i=1,...,\abs{m}}.
\end{align*}
Thus $K_B$ is a space of rational functions $f$ of the form
\begin{align*}
f(z)=\frac{p(z)}{\prod_i(1-\bar{\lambda_i}z)},
\end{align*}
where $p(z)$ is a polynomial of degree at most $\abs{m}-1$. If the zeros of $B$ are not distinct the above remains valid but the Cauchy kernels have to be replaced by
\begin{align*}
\frac{z^{k-1}}{(1-\bar{\lambda}_iz)^{k}},\quad 1\leq k\leq k_i,
\end{align*}
where $k_i$ denotes the multiplicity of $\lambda_i$. In our consecutive proofs, however, we omit this case and assume that $A$ is diagonalizable. This does not result in any difficulties since upper bounds obtained in the special case extend by continuity to bounds for non-diagonalizable matrices. The assumption that $A$ can be diagonalized is not principal; virtually all computations in the manuscript can be carried out in the more general case. We avoid non-diagonalizable $A$ and
rely on continuity only for notational convenience.

One natural orthonormal basis for $K_B$ is the Malmquist-Walsh basis $\{e_k\}_{k=1,...,\abs{m}}$ with {(\cite{TS}, page 117)}
\begin{align*}
e_k(z):=\frac{(1-\abs{\lambda_k}^2)^{1/2}}{1-\bar{\lambda}_k z}\prod_{i=1}^{k-1}\frac{z-\lambda_i}{1-\bar{\lambda}_iz},
\end{align*}
where, as it will remain throughout the manuscript, the empty product is defined to be $1$ i.e.
\begin{align*}
e_1(z)=\frac{(1-\abs{\lambda_1}^2)^{1/2}}{1-\bar{\lambda}_1z}.
\end{align*}

The \emph{model operator} $M_B$ acts on $K_B$ as
\begin{align*}M_B:\:K_B&\rightarrow K_B\\
f&\mapsto M_B(f):=P_B(zf),
\end{align*}
where $P_B$ denotes the orthogonal projection on $K_B$. In other words, $M_B$ is the compression of the multiplication operation by $z$ to the model space $K_B$ (see \cite{TS} for a detailed discussion of model operators and spaces). As multiplication by $z$ has operator norm $1$ it is clear that $M_B$ is a Hilbert space contraction. Moreover, it is not hard to show that the eigenvalues of $M_B$ are exactly the zeros of the corresponding Blaschke product (see \cite{OF}, page 228 and Proposition~\ref{modelcomp} in the article at hand).

\subsection{Spectral bounds on the norm of a function of a matrix}\label{specnr}

This subsection contains a brief outline of methods to obtain spectral bounds on a function of a matrix. For a more detailed account see \cite{AC,TS,OF} and the references therein. Suppose that $f$ is holomorphic on a domain containing all eigenvalues of $A$ and let $\gamma$ be a smooth curve in this domain that encloses the eigenvalues. The matrix $f(A)$ is defined by the Dunford-Taylor integral~\cite{kato}
\begin{align*}
f(A)=\frac{1}{2\pi i}\int_\gamma f(\zeta)(\zeta-A)^{-1}\d\zeta.
\end{align*} 

It is easily seen that if $f(z)=\sum_{k=0}^l a_kz^k$ is a polynomial then $f(A)=\sum_{k=0}^l a_kA^k$ and that the correspondence $f\mapsto f(T)$ is an algebra homomorphism from the algebra of holomorphic functions (on the given domain) to $\cM_n$ i.e $(f+g)(T)=f(T)+g(T)$ and $(fg)(T)=f(T)g(T)$~(see \cite{kato}, Chapter I.6). A unital Banach algebra $\cA$ with elements in $Hol(\mathbb{D})$ will be called a function algebra if

 i) $\cA$ contains all polynomials and $\lim_{n\rightarrow\infty}\Norm{z^n}{\cA}^{1/n}=1$ and 

ii) $(f\in \cA,\:\lambda\in\mathbb{D},\: f(\lambda)=0)$ implies that  $\frac{f}{z-\lambda}\in \cA$. \newline{}
%
%
Following the conventions of \cite{AC} we say that a set of matrices $\Gamma$ obeys an $\cA$ functional calculus with constant $C$ if 
\begin{align*}
\Norm{f(A)}{}\leq C\Norm{f}{\cA},
\end{align*}
holds for any $A\in\Gamma$ and $f\in\cA$. Here $\Norm{f}{\cA}$ denotes the norm of $f$ in $\cA$.
%
%
%
Clearly, this is only possible if all eigenvalues of $A$ are contained in $\bar{\mathbb{D}}$.
For us, two instances of such inequalities will be important. In the first example we consider Hilbert space contractions, while the second one treats power-bounded Banach space operators.

i) The family of Hilbert space contractions $\Gamma=\{A\in \cM_n| \Norm{A}{\infty}\leq1\}$ is related to an $H_\infty$ functional calculus, since by von Neumann\rq{}s inequality \cite{Nagy,LOTS} we have for any $f$ in the disk algebra $H_\infty\cap\cC(\bar{\mathbb{D}})$ (the set of bounded holomorphic functions on $\mathbb{D}$ that admit a continuous extension to the boundary) and $A\in\Gamma$ with $\sigma(A)\subset\mathbb{D}$
\begin{align*}
\Norm{f(A)}{\infty}\leq\Norm{f}{H_\infty}.
\end{align*}

ii) Consider a family $\Gamma=\{A\in \cM_n| \Norm{A^k}{}\leq C\ \forall k\in\mathbb{N}\}$ of Banach space operators that are power bounded by some constant $C<\infty$. This family admits a Wiener algebra functional calculus since for any $f\in W$ and $A\in\Gamma$
\begin{align*}
\Norm{f(A)}{}=\Norm{\sum_{k\geq0}\hat{f}(k)A^k}{}\leq \sum_{k\geq0}\abs{\hat{f}(k)}\Norm{A^k}{}\leq C\sum_{k\geq0}\abs{\hat{f}(k)}=C\Norm{f}{W}
\end{align*}
holds.\\

At first glance, the inequalities of $i)$ and $ii)$ seem to be of little use when it comes to finding spectral bounds on $\Norm{f(A)}{}$ since the obtained upper bounds do not depend on $A$ anymore. To obtain better estimates one can rely on the following insight.  Instead of considering the function $f$ directly, we add multiples of $m$ (or any other annihilating polynomial) to this function and consider $h=f+mg,\:g\in \cA$ instead of $f$. It is immediate that $\Norm{f(X)}{}=\Norm{h(X)}{}$. The following simple but crucial lemma summarizes this point:

{\begin{lemma}[\cite{AC} Lemma 3.1]\label{wiener} Let $m\neq0$ be a polynomial and let $\Gamma$ be a set of matrices that obey an $\cA$ functional calculus with constant $C$ and that satisfy $m(A)=0\ \forall A\in\Gamma$. Then
%
%
\begin{align*}
\Norm{f(A)}{}\leq C \Norm{f}{\cA/m\cA},\ \forall A\in\Gamma,
\end{align*}
where $\Norm{f}{\cA/m\cA}=\inf{\{\norm{h}{\cA}|\ h=f+mg,\:g\in \cA\}}$.
\end{lemma}}
\begin{proof} For any $g\in A$ we have that $\Norm{f(A)}{}=\Norm{(f+mg)(A)}{}\leq C \Norm{f+mg}{\cA}$.
\end{proof}
If $\sigma(A)\subset\mathbb{D}$ (and $A$ can be diagonalized) it follows directly from the definition of the function algebra (see also \cite{AC}, Section 3.1 (iii) or \cite{effe}, Section 1.2 P4) that 
$$\Norm{f}{\cA/m\cA}=\inf\{\Norm{g}{\cA}\:|\:g\in\cA,\:g(\lambda_i)=f(\lambda_i)\:\forall\lambda_i\in\sigma(A)\},$$
which is a Nevanlinna-Pick type interpolation problem in $\cA$.
If the eigenvalue $\lambda_i$ carries a multiplicity $k_i>1$ in $m$ the above remains valid but at $\lambda_i$ the first $k_i-1$ derivatives of $f$ and $g$ must coincide.
Since for $\sigma(A)\subset{\mathbb{D}}$ the Blaschke product is holomorphic on a set containing $\bar{\mathbb{D}}$ we can define $\Norm{f}{\cA/B\cA}$ as in Lemma~\ref{wiener} and note (\cite{effe}, Lemma 3.1) that as before
$$\Norm{f}{\cA/B\cA}=\inf\{\Norm{g}{\cA}\:|\:g\in\cA,\:g(\lambda_i)=f(\lambda_i)\:\forall\lambda_i\in\sigma(A)\}.$$
In the special case $\cA=H_\infty$ it is possible to evaluate $\Norm{f}{H_\infty/BH_\infty}$ using Sarason's approach to the Nevanlinna-Pick problem~\cite{Sarason,interpol} or the Commutant lifting theorem of B.~Sz.-Nagy and C.~Foia\c{s}~\cite{FF,NagyF,interpol}.
\begin{lemma}[\cite{AC}~Theorem 3.12,~\cite{OF}~Theorem 3.1.11]\label{sarasonnorm}
For any $f\in H_\infty$ it holds that
\begin{align*}
\Norm{f}{H_\infty/BH_\infty}=\Norm{f(M_B)}{\infty}.
\end{align*}
\end{lemma}

\section{Hilbert space contractions}\label{contractions}
Spectral bounds on the resolvent of a Hilbert space contraction were derived in \cite{BS}. The authors provide an upper bound in terms of a certain Toeplitz matrix, compute the norm of this matrix and present a sequence of matrices that approaches their upper bound. The following theorem summarizes the basic three assertions from the discussion of Hilbert spaces contractions in \cite{BS}.
{\begin{theorem}[\cite{BS,LOTS}]\label{SD} \leavevmode
\begin{enumerate}
\item Let $A$ be an $n\times n$ matrix with $\norm{A}{\infty}\leq1$ and $1\notin\sigma(A)$. Then 
\begin{align*}
\Norm{(\id-A)^{-1}}{\infty}\leq\frac{\Norm{M_n}{\infty}}{\min_{\lambda_i\in\sigma{(A)}}\abs{1-\lambda_i}},
\end{align*}
with the $n\times n$ matrix 
\begin{align*}
M_n:=\begin{pmatrix}
 1 & 0 & \hdots &  0\\
2 & 1 & \ddots & \vdots\\
\vdots & \ddots & \ddots & 0\\
2 & \hdots & 2 & 1
\end{pmatrix}.
\end{align*}
\item
It holds that $\Norm{M_n}{\infty}=\cot{(\frac{\pi}{4n})}$.
\item For any $a\in(0,1)$ there are $n\times n$ matrices $A_n(a)$ with $\Norm{A_n(a)}{\infty}\leq1$ and $\sigma(A)=\{a\}$ such that
\begin{align*}\lim_{a\to1}(1-a)(\id-A_n(a))^{-1}=M_n.\end{align*}
\end{enumerate}
\end{theorem}}

In this paper we recover the statements \textit{1}~and \textit{3}~using a unified approach based on the techniques developed in \cite{AC}. Our strategy is to directly compute and bound the entries of the model operator in Malmquist-Walsh basis. Our approach has the advantage that it yields spectral bounds for any $\zeta\in\mathbb{C}-\sigma(A)$ and that the optimality statement \textit{3}~is automatic. Concerning the second point of the theorem we present a technique going back to \cite{eger} in order to compute the norm of Toeplitz matrices of the form
\begin{align}
M_n(\beta):=\begin{pmatrix}
 1 & 0 & \hdots &  0\\
\beta & 1 & \ddots & \vdots\\
\vdots & \ddots & \ddots & 0\\
\beta & \hdots & \beta & 1
\end{pmatrix},\qquad\beta\in[0,2]\label{toeplitz}.
\end{align}
%

\begin{theorem}\label{improve1} Let $A$ be an $n\times n$ matrix with $\norm{A}{\infty}\leq1$ and minimal polynomial
$m=\prod_{i=1}^{\abs{m}}(z-\lambda_i)$ with $\sigma(A)\subset\mathbb{D}$. Then for any $\zeta\in\mathbb{C}-\sigma(A)$ it holds that
$$
\Norm{(\zeta-A)^{-1}}{\infty}\leq\Norm{(\zeta-M_B)^{-1}}{\infty}
$$ and
$$\left((\zeta-M_B)^{-1}\right)_{ij}=\begin{cases}\ \ 0\ &if \ i<j\\
\frac{1}{\zeta-\lambda_i}\ &if\ i=j\\
\frac{(1-\abs{\lambda_i}^2)^{1/2}}{\zeta-\lambda_i}\frac{(1-\abs{\lambda_j}^2)^{1/2}}{\zeta-\lambda_j}\prod_{\mu=j+1}^{i-1}\left(\frac{1-\bar{\lambda}_\mu\zeta}{\zeta-\lambda_\mu}\right)\ &if \ i>j
\end{cases}$$
with respect to the Malmquist-Walsh basis. (The empty product is defined to be $1$.)
\end{theorem}

To compare our new result Theorem~\ref{improve1} to Theorem~\ref{SD} we note that for any $n\times n$ matrices $A=(a_{ij})$ and $B=(b_{ij})$, the condition $\abs{a_{ij}}\leq b_{ij}\ \forall i,j$ implies that $\Norm{A}{\infty}\leq\Norm{B}{\infty}$. Suppose for instance that $\abs{\zeta}\leq1$. Then we can estimate the off-diagonal components of $(\zeta-M_B)^{-1}$ by
\begin{align*}
&\Abs{\frac{(1-\abs{\lambda_i}^2)^{1/2}}{1-\bar{\lambda}_i\zeta}\frac{(1-\abs{\lambda_j}^2)^{1/2}}{1-\bar{\lambda}_j\zeta}\prod_{\mu=j}^{i}
\left(\frac{1-\bar{\lambda}_\mu\zeta}{\zeta-\lambda_\mu}\right)}
\leq\max_i\frac{1-\abs{\lambda_i}^2}{\abs{1-\bar{\lambda}_i\zeta}^2}\prod_{\mu=1}^{\abs{m}}
\Abs{\frac{1-\bar{\lambda}_\mu\zeta}{\zeta-\lambda_\mu}}\\
&\leq\max_{i}\frac{1}{\abs{1-\bar{\lambda}_i\zeta}}\max_{i}
\frac{1-\abs{\lambda_i}^2}{\abs{1-\bar{\lambda}_i\zeta}}
\prod_{\mu=1}^{\abs{m}}
\Abs{\frac{1-\bar{\lambda}_\mu\zeta}{\zeta-\lambda_\mu}}
\leq\max_{i}\frac{2}{\abs{1-\bar{\lambda}_i\zeta}}\prod_{\mu=1}^{\abs{m}}\Abs{\frac{1-\bar{\lambda}_\mu\zeta}{\zeta-\lambda_\mu}},
\end{align*}
which yields the component-wise estimate
\begin{align*}
\Abs{\left((\zeta-M_B)^{-1}\right)_{ij}}\leq\frac{1}{\min_{\lambda_k\in\sigma(A)}
\abs{1-\bar{\lambda}_k\zeta}}\prod_{\mu=1}^{\abs{m}}\Abs{\frac{1-\bar{\lambda}_\mu\zeta}{\zeta-\lambda_\mu}}\cdot\begin{cases} 0\ &if \ i<j\\
1\ &if\ i=j\\
2\ &if \ i>j
\end{cases}.
\end{align*}
\begin{corollary}\label{corollary1}
Under the assumptions of Theorem~\ref{improve1} suppose that $\abs{\zeta}\leq1$. It follows
\begin{align*}
\Norm{(\zeta-A)^{-1}}{\infty}\leq\frac{\Norm{M_{\abs{m}}}{\infty}}{\min_{\lambda_k\in\sigma(A)}
\abs{1-\bar{\lambda}_k\zeta}}\frac{1}{\Abs{B(\zeta)}},
\end{align*}
where $B(\zeta)=\prod_{i=1}^{\abs{m}}\frac{\zeta-\lambda_i}{1-\bar{\lambda}_i\zeta}$ is the Blaschke product associated with $m$.
\end{corollary}
We can pass to the general case $\sigma(A)\subset\overline{\mathbb{D}}$ by continuous extension.
Setting $\zeta=1$ Corollary~\ref{corollary1} is the first assertion of Theorem~\ref{SD} with the bonus that on the right hand side the norm of an $\abs{m}\times\abs{m}$ matrix occurs (compare \cite{BS} Section 6 B). 
However, if $\max_{i}\frac{1-\abs{\lambda_i}^2}{\abs{1-\lambda_i}}=\beta$ is given we have (with the same computation as above) for $\zeta=1$
\begin{align*}
\Abs{\left((\id-M_B)^{-1}\right)_{ij}}\leq\frac{1}{\min_{\lambda_k\in\sigma(A)}\abs{1-\lambda_k}}\cdot\begin{cases} 0\ &if \ i<j\\
1\ &if\ i=j\\
\beta\ &if \ i>j
\end{cases}
\end{align*}
and we can improve the bound in Theorem~\ref{SD} if we can compute $\Norm{M_n(\beta)}{\infty}$ (see~\eqref{toeplitz}).
The following theorem generalizes the discussion of Toeplitz matrices in \cite{BS}. It establishes an indirect possibility to compute $\Norm{M_n(\beta)}{\infty}$.
\begin{proposition}\label{Top} Let $M_n(\beta)$ with $\beta\in(0,2]$ be the $n\times n$ Toeplitz matrix introduced in~\eqref{toeplitz}. Then the equation
\begin{align}
\beta\cot{(n\theta)}+(2-\beta)\cot{(\theta/2)}=0,\qquad\theta\in\mathbb{R}\label{cotan}
\end{align}
has a unique solution $\theta^*\in[\frac{2n-1}{2n}\pi,\pi)$ and 
\begin{align*}
\Norm{M_n(\beta)}{\infty}=\frac{1}{2}\sqrt{(\beta-2)^2+\frac{\beta^2}{\cot^2{(\theta^*/2)}}}.
\end{align*}
In particular it holds that
$\Norm{M_n(0)}{\infty}=1$ and $\Norm{M_n(1)}{\infty}=\frac{1}{2\sin(\frac{\pi}{4n+2})}$ and $\Norm{M_n(2)}{\infty}=\cot{(\frac{\pi}{4n})}$.
\end{proposition}
It is possible to expand $\cot(n\theta)$ in Equation~\eqref{cotan} in terms of $\cot(\theta/2)$, which yields a polynomial equation in $\cot(\theta/2)$. Since $\Norm{M_n(\beta)}{\infty}$ only depends on  $\cot(\theta/2)$ (and $\beta$) computing $\Norm{M_n(\beta)}{\infty}$ is equivalent to finding the unique zero of the resulting polynomial in the interval $(0,\cot\left(\frac{2n-1}{4n}\pi\right)]$ as a function of $\beta$.

Finally, statement \textit{3}~of Theorem~\ref{SD} can be recovered from Theorem~\ref{improve1} with the choice of a minimal polynomial $m=(z-a)^n$, $a\in(0,1)$ and setting $A_n(a)=M_B$. In this case we have for $1\leq i,j\leq n$ that
\begin{align*}
\Abs{\left((\id-M_B)^{-1}\right)_{ij}}=\frac{1}{1-a}\cdot\begin{cases}\ 0\ &if \ i<j\\
\ 1\ &if\ i=j\\
1+a\ &if \ i>j.
\end{cases}
\end{align*}
Letting $a\to1$ proves item~\textit{3}~of Theorem~\ref{SD}. In the following Subsection~\ref{model} we compute the entries of $M_B$ with respect to the Malmquist-Walsh basis. This yields a simple form for matrices that achieve equality in Theorem~\ref{improve1} i.e.~for $A$ with largest $\Norm{(\zeta-A)^{-1}}{\infty}$ for a given spectrum.
\begin{proposition}\label{modelcomp} The components of the model operator $M_B$ with respect to Malmquist-Walsh basis are given by
$$\left(M_B\right)_{ij}=\begin{cases}\qquad\qquad\qquad{0}\ &if \ i<j\\
\qquad\qquad\qquad\lambda_i\ &if\ i=j\\
(1-\abs{\lambda_i}^2)^{1/2}(1-\abs{\lambda_j}^2)^{1/2}\prod_{\mu=j+1}^{i-1}\left(-\bar{\lambda}_\mu\right)\ &if \ i>j.
\end{cases}$$
\end{proposition}
Hence, one explicit form of the matrices $A_n(a)$ in Theorem~\ref{SD} is
\begin{align*}
A_n(a):=\begin{pmatrix}
 a & 0 & \hdots &  \hdots &0\\
1-a^2 & a & \ddots & & \vdots\\
-a(1-a^2)& 1-a^2& a & \ddots& \vdots\\
\vdots & \ddots & \ddots &\ddots& 0\\
(-1)^na^{(n-2)}(1-a^2) & \hdots &-a(1-a^2) &1-a^2 & a
\end{pmatrix}.
\end{align*}

Finally, we note that Theorem~\ref{improve1} is stronger than Theorem~\ref{SD} in that it holds for \emph{general} $\zeta$ and yields an \emph{optimal} bound for \emph{general} spectra.\\

The rest of this section is organized in two subsections. The first, Subsection~\ref{model}, contains a proof of Theorem~\ref{improve1} and Proposition~\ref{modelcomp} while in Subsection~\ref{sec:top} we prove Proposition~\ref{Top}.

\subsection{A model operator approach to resolvent bounds}\label{model}
As mentioned before our approach is to bound a function of a matrix in terms of the norm of a representative function. A key role is played by Lemma~\ref{wiener}, which however requires that $f\in\cA$. In order to derive upper bounds for rational functions such as the resolvent we need to extend Lemma~\ref{wiener}.
The following is based on the techniques of \cite{AC}, Lemma~3.2 for the discussion of inverses. Here, we present an extension, which is adapted to our purposes.
\begin{lemma}\label{ratmod} Let $A$ be an $n\times n$ matrix with $\sigma(A)\subset\mathbb{D}$ and let $\psi$ be a rational function with poles $(\xi_i)_{i=1,...,k}$ such that $\bigcup_i\{\xi_i\}\cap\sigma(A)=\emptyset$.
\begin{enumerate}
\item If $A$ obeys an $\cA$-functional calculus with constant $C$ then $\Norm{\psi(A)}{}\leq C\inf\{\Norm{g}{\cA}\:|\:g\in\cA,\ g(\lambda_i)=\psi(\lambda_i)\ i=1,...,n\}$.
\item If $\norm{A}{\infty}\leq1$ holds then $\norm{\psi(A)}{\infty}\leq\norm{\psi(M_B)}{\infty}$.
\end{enumerate}
\end{lemma}
\begin{proof} We extend Lemma~\ref{wiener} to the situation, when $\psi$ is rational. Define $\varphi:=\psi\cdot\prod_{j=1}^{k}\left(\frac{m(\xi_j)-m}{m(\xi_j)}\right)^{k_j}$, where $k_j$ denotes the multiplicity of the pole at $\xi_j$ and note that $\varphi$ is polynomial and that $\psi(A)=\varphi(A)$. It follows using Lemma~\ref{wiener} that
\begin{align*}
\Norm{\psi(A)}{}&=\Norm{\varphi(A)}{}\leq C\Norm{\varphi}{\cA/m\cA}=C\inf\{\Norm{g}{\cA}\:|\:g\in\cA,\ g(\lambda_i)=\varphi(\lambda_i)\ i=1,...,n\}\\
&=C\inf\{\Norm{g}{\cA}\:|\:g\in \cA,\ g(\lambda_i)=\psi(\lambda_i)\ i=1,...,n\},
\end{align*}
which proves the first assertion. For the second one we consider the same $\varphi$ as above and note that
\begin{align*}
\Norm{\psi(A)}{\infty}&=\Norm{\varphi(A)}{\infty}\leq\Norm{\varphi}{H_\infty/BH_\infty}=\Norm{\varphi(M_B)}{\infty},
\end{align*}
where we applied Lemma~\ref{sarasonnorm} in the last step. But as $m(M_B)=0$ it follows that $\varphi(M_B)=\psi(M_B)$.
\end{proof}
Let us remark that Lemma~\ref{ratmod} remains valid if the eigenvalue $\lambda_i$ carries degeneracy $k_i$ in $m$. The point here is to replace the $\inf$ on the right hand side of \emph{1} with $\inf\{\Norm{g}{\cA}\:|\:g\in\cA,\ g^{(k)}(\lambda_i)=\psi^{(k)}(\lambda_i)\ , 0\leq k<k_i\}$, where the superscript $k$ denotes the $k$-th derivative.

\begin{lemma}\label{combi1} Let $\{\lambda_i\}_{i=1,...,n}\subset\mathbb{D}$ and let $\zeta\in\mathbb{C}-\{\lambda_i\}_{i=1,...,n}$ and $j<i$ then
$$\sum_{\mu=j}^{i}\frac{1}{\zeta-\lambda_\mu}\frac{\prod_{\nu:\nu\neq i,\nu\neq j}(1-\bar{\lambda}_\nu\lambda_\mu)}{\prod_{\nu:\nu\neq\mu}(\lambda_\mu-\lambda_\nu)}=\frac{1}{(1-\bar{\lambda}_i\zeta)(1-\bar{\lambda}_j\zeta)}\:\prod_{\mu=j}^{i}\left(\frac{1-\bar{\lambda}_\mu\zeta}{\zeta-\lambda_\mu}\right).$$
\end{lemma}
%
%
%
%
%
%
%
%
%
\begin{proof}[Proof of Lemma~\ref{combi1}] We present two proofs for this lemma. The first one arises naturally in the context of $H_2$ spaces (see the proof of Theorem~\ref{improve1}), while the second one is more direct and simple.  We define $t(z):=\frac{z}{\zeta-z}\frac{1}{(1-\bar{\lambda}_iz)(1-\bar{\lambda}_jz)}$ and the (truncated) Blaschke product $B_{ji}(z):=\prod_{\mu=j}^{i}\frac{z-\lambda_\mu}{1-\bar{\lambda}_\mu z}$ and compute the $L_2(\partial\mathbb{D})$ scalar product. Suppose for now that $\abs{\zeta}>1$ then
\begin{align*}
&\braket{t}{B_{ji}}=\int_0^{2\pi}t(z)\overline{B_{ji}(z)}\Big|_{z=e^{i\phi}}\frac{\d\phi}{2\pi}=\int_0^{2\pi}t(z)\prod_\mu\frac{1-\bar{\lambda}_{\mu}z}{z-\lambda_\mu}\Big|_{z=e^{i\phi}}\frac{\d\phi}{2\pi}\\
&=\frac{1}{2\pi i}\int_{\partial\mathbb{D}}\frac{1}{(\zeta-z)(1-\bar{\lambda}_iz)(1-\bar{\lambda}_jz)}\prod_\mu\frac{1-\bar{\lambda}_{\mu}z}{z-\lambda_\mu}\:\d z
=\sum_{\mu=j}^{i}\frac{1}{\zeta-\lambda_\mu}\frac{\prod_{\nu:\nu\neq i,\nu\neq j}(1-\bar{\lambda}_\nu\lambda_\mu)}{\prod_{\nu:\nu\neq\mu}(\lambda_\mu-\lambda_\nu)},
\end{align*}
where in the last step we applied the Residue theorem and made use of the assumption $\abs{\zeta}>1$. On the other hand
\begin{align*}
\braket{B_{ji}}{t}&=\int_0^{2\pi}B_{ji}(z)\overline{t(z)}\Big|_{z=e^{i\phi}}\frac{\d\phi}{2\pi}=\frac{1}{2\pi i}\int_{\partial \mathbb{D}}\prod_\mu\frac{z-\lambda_\mu}{1-\bar{\lambda}_{\mu}z}\frac{1}{\bar{\zeta}z-1}\frac{z}{(z-\lambda_i)(z-\lambda_j)}\d z\\
&=\frac{1}{(1-\lambda_i\bar{\zeta})(1-\lambda_j\bar{\zeta})}\:\prod_\mu\frac{1-\lambda_{\mu}\bar{\zeta}}{\bar{\zeta}-\bar{\lambda}_\mu}.
\end{align*}
Clearly, $\braket{t}{B_{ji}}=\overline{\braket{B_{ji}}{t}}$ from which the lemma follows for
$\abs{\zeta}>1$. In case that 
$\abs{\zeta}<1$ we compute similarly
\begin{align*}
\braket{t}{B_{ji}}&=\frac{1}{2\pi i}\int_{\partial\mathbb{D}}\frac{1}{(\zeta-z)(1-\bar{\lambda}_iz)(1-\bar{\lambda}_jz)}\prod_\mu\frac{1-\bar{\lambda}_{\mu}z}{z-\lambda_\mu}\:\d z\\
&=\sum_{\mu=j}^{i}\frac{1}{\zeta-\lambda_\mu}\frac{\prod_{\nu:\nu\neq i,\nu\neq j}(1-\bar{\lambda}_\nu\lambda_\mu)}{\prod_{\nu:\nu\neq\mu}(\lambda_\mu-\lambda_\nu)}-\frac{1}{(1-\bar{\lambda}_i\zeta)(1-\bar{\lambda}_j\zeta)}\:\prod_\mu\frac{1-\bar{\lambda}_{\mu}\zeta}{\zeta-\lambda_\mu}
\end{align*}
and
\begin{align*}
\braket{B_{ji}}{t}=0.
\end{align*}
The case $\abs{\zeta}=1$ follows by continuity. For the second proof we multiply both sides of the lemma with $\prod_{\mu=j}^i(\zeta-\lambda_\mu)$ to obtain a polynomial equation in $\zeta$
\begin{align*}
\sum_{\mu=j}^{i}\prod_{\nu:\nu\neq\mu}(\zeta-\lambda_\nu)\frac{\prod_{\nu:\nu\neq i,\nu\neq j}(1-\bar{\lambda}_\nu\lambda_\mu)}{\prod_{\nu:\nu\neq\mu}(\lambda_\mu-\lambda_\nu)}=\prod_{\mu:\mu\neq i,\mu\neq j}(1-\bar{\lambda}_\mu\zeta).
\end{align*}
The polynomial on the left hand side has degree at most $i-j$ and the degree of the polynomial on the right hand side is $i-j-1$. Two polynomials of a certain degree $n$ are the same if and only if they coincide at $n+1$ nodes.
We choose the $i-j+1$ values $\{\lambda_\alpha\}_{j\leq\alpha\leq i}$ and verify that for this choice equality indeed holds:
\begin{align*}
\sum_{\mu=j}^{i}\prod_{\nu:\nu\neq\mu}(\zeta-\lambda_\nu)\frac{\prod_{\nu:\nu\neq i,\nu\neq j}(1-\bar{\lambda}_\nu\lambda_\mu)}{\prod_{\nu:\nu\neq\mu}(\lambda_\mu-\lambda_\nu)}\Bigg|_{\zeta=\lambda_\alpha}&=\prod_{\nu:\nu\neq\alpha}(\lambda_\alpha-\lambda_\nu)\frac{\prod_{\nu:\nu\neq i,\nu\neq j}(1-\bar{\lambda}_\nu\lambda_\alpha)}{\prod_{\nu:\nu\neq\alpha}(\lambda_\alpha-\lambda_\nu)}\\
&=\prod_{\nu:\nu\neq i,\nu\neq j}(1-\bar{\lambda}_\nu\lambda_\alpha).
\end{align*}
\end{proof}
We are now ready to present a proof of Theorem~\ref{improve1}.
\begin{proof}[Proof of Theorem~\ref{improve1}]
The first assertion follows directly from Lemma~\ref{ratmod}. To compute the matrix entries of $(\zeta-M_B)^{-1}$ with respect to Malmquist-Walsh basis we recall that 
\begin{align*}
(\zeta-M_B)^{-1}=\varphi(M_B),
\end{align*}
where $\varphi(z):=\frac{1}{\zeta-z}\:\frac{m(\zeta)-m(z)}{m(\zeta)}$ is a polynomial. We have that
\begin{align}
&((\zeta-M_B)^{-1})_{ij}=\braket{\varphi(M_B)e_j}{e_i}=\braket{P_B(\varphi\: e_j)}{e_i}=\braket{\varphi\: e_j}{e_i}=\int_0^{2\pi}\varphi(z)e_j(z)\overline{e_i(z)}\Big|_{z=e^{i\phi}}\frac{\d \phi}{2\pi}\nonumber\\
&=\frac{((1-\abs{\lambda_i}^2)(1-\abs{\lambda_j}^2))^{1/2}}{2\pi i}\int_{\partial\mathbb{D}}\varphi(z)\frac{1}{(1-\bar{\lambda}_iz)(1-\bar{\lambda}_jz)}\prod_{\mu=1}^{j-1}\frac{z-\lambda_\mu}{1-\bar{\lambda}_\mu z}\prod_{\nu=1}^{i}\frac{1-\bar{\lambda}_\nu z}{z-\lambda_\nu}\:\d z.\label{cases}
\end{align}
In case that $j>i$ the integrand is holomorphic on $\mathbb{D}$. Hence, the integral in \eqref{cases} is zero. If $j=i$
we have that
\begin{align*}
\frac{(1-\abs{\lambda_i}^2)}{2\pi i}\int_{\partial\mathbb{D}}\frac{1}{\zeta-z}\:\frac{m(\zeta)-m(z)}{m(\zeta)}\frac{1}{(1-\bar{\lambda}_iz)(z-\lambda_i)}\:\d z=\frac{1}{\zeta-\lambda_i}.
\end{align*}
Finally if $j<i$ then \eqref{cases} becomes
\begin{align*}
&\frac{((1-\abs{\lambda_i}^2)(1-\abs{\lambda_j}^2))^{1/2}}{2\pi i}\int_{\partial\mathbb{D}}\frac{1}{\zeta-z}\:\frac{m(\zeta)-m(z)}{m(\zeta)}\frac{1}{(1-\bar{\lambda}_iz)(1-\bar{\lambda}_jz)}\prod_{\nu=j}^{i}\frac{1-\bar{\lambda}_\nu z}{z-\lambda_\nu}\:\d z\\
&=((1-\abs{\lambda_i}^2)(1-\abs{\lambda_j}^2))^{1/2}\sum_{\mu=j}^{i}\frac{1}{\zeta-\lambda_\mu}\frac{\prod_{\nu:\nu\neq i,\nu\neq j}(1-\bar{\lambda}_\nu\lambda_\mu)}{\prod_{\nu:\nu\neq\mu}(\lambda_\mu-\lambda_\nu)}.
\end{align*}
An application of Lemma~\ref{combi1} concludes the proof of Theorem~\ref{improve1}.
\end{proof}
Proposition~\ref{modelcomp} is verified via a direct calculation.
\begin{proof}[Proof of Proposition~\ref{modelcomp}] We proceed as in the derivation of Theorem~\ref{improve1} and conclude
\begin{align*}
(M_B)_{ij}=((1-\abs{\lambda_i}^2)(1-\abs{\lambda_j}^2))^{1/2}\int_0^{2\pi}\frac{z^2}{(1-\bar{\lambda}_iz)(1-\bar{\lambda}_jz)}\prod_{\mu=1}^{j-1}\frac{z-\lambda_\mu}{1-\bar{\lambda}_\mu z}\prod_{\nu=1}^{i}\frac{1-\bar{\lambda}_\nu z}{z-\lambda_\nu}\:\Bigg|_{z=e^{i\phi}}\frac{\d \phi}{2\pi}.
\end{align*}
If $j>i$ the Residue theorem reveals that the integral is zero. Similarly, if $i=j$ the integral is given by $\lambda_i$. Finally if $i>j$ we compute
\begin{align*}
&\int_0^{2\pi}\frac{z^2}{(1-\bar{\lambda}_iz)(1-\bar{\lambda}_jz)}\prod_{\mu=j}^{i}\frac{1-\bar{\lambda}_\mu z}{z-\lambda_\mu}\:\Bigg|_{z=e^{i\phi}}\frac{\d \phi}{2\pi}=\overline{\int_0^{2\pi}\frac{1}{(z-\lambda_i)(z-\lambda_j)}\prod_{\mu=j}^{i}\frac{z-\lambda_\mu}{1-\bar{\lambda}_\mu z}\:\Bigg|_{z=e^{i\phi}}\frac{\d \phi}{2\pi}}\\
&=\overline{\frac{1}{2\pi i}\int_{\partial\mathbb{D}}\frac{1}{z(z-\lambda_i)(z-\lambda_j)}\prod_{\mu=j}^{i}\frac{z-\lambda_\mu}{1-\bar{\lambda}_\mu z}\d z}=\overline{\prod_{\mu=j+1}^{i-1}(-\lambda_\mu)},
\end{align*}
where the last step again uses the Residue theorem.
\end{proof}

\subsection{Computing the norm of certain Toeplitz matrices}\label{sec:top}
In this subsection we prove Proposition~\ref{Top} with a direct computation of $\Norm{M_n(\beta)}{\infty}$. Our approach is guided by the techniques developed in \cite{eger}. The quantities $\Norm{M_n(1)}{\infty}$ and $\Norm{M_n(2)}{\infty}$ are computed in \cite{BS} and \cite{LOTS} (Lemma 9.6.5) following a different approach.
\begin{proof}[Proof of Proposition~\ref{Top}]
Instead of working with 
\begin{align*}
M_n(\beta)=\begin{pmatrix}
 1 & 0 & \hdots &  0\\
\beta & 1 & \ddots & \vdots\\
\vdots & \ddots & \ddots & 0\\
\beta & \hdots & \beta & 1
\end{pmatrix}
\end{align*}
directly, we consider the matrix
\begin{align*}
\tilde{M}_n(\beta):=\begin{pmatrix}
\beta  &\hdots& \beta &  1\\
\vdots& \reflectbox{$\ddots$} & 1 & 0\\
\beta & \reflectbox{$\ddots$} & \reflectbox{$\ddots$} & \vdots\\
1& 0 & \hdots & 0
\end{pmatrix}
\end{align*}
and note that 
\begin{align*}
\Norm{M_n(\beta)}{\infty}=\Norm{\tilde{M}_n(\beta)}{\infty}.
\end{align*}
As $\tilde{M}_n(\beta)$ is Hermitian all its eigenvalues are real and its $\infty$-norm is simply the largest in magnitude eigenvalue. The eigenvalues of $\tilde{M}_n(\beta)^2$ are the eigenvalues of $\tilde{M}_n(\beta)$ squared. Hence, we are looking for the largest $\lambda^2$ such that
\begin{align*}
0=\det{(\tilde{M}_n(\beta)^2-\lambda^2\id)}=\det{(\tilde{M}_n(\beta)-\lambda\id)(\tilde{M}_n(\beta)+\lambda\id)}.
\end{align*}
Direct computation reveals that
\begin{align*}
&(\tilde{M}_n(\beta)-\lambda\id)(\tilde{M}_n(\beta)+\lambda\id)=\textnormal{\small{$\begin{pmatrix}
\beta-\lambda  &\beta &\beta &\hdots& \beta &  1\\
\beta  &\beta-\lambda & & & 1 &  0\\
\beta&   &  & & & 0\\
\vdots& &   & & & \vdots\\
\beta & 1&   & & -\lambda& 0\\
1& 0 &0&\hdots & 0&-\lambda
\end{pmatrix}\cdot\begin{pmatrix}
\beta+\lambda  &\beta &\beta &\hdots& \beta &  1\\
\beta  &\beta+\lambda & & & 1 &  0\\
\beta&   &  & & & 0\\
\vdots& &   & & & \vdots\\
\beta & 1&   & & \lambda& 0\\
1& 0 &0&\hdots & 0&\lambda
\end{pmatrix}$}}\\
&=\begin{pmatrix}
(n-1)\beta^2-\lambda^2+1  &(n-2)\beta^2+\beta &(n-3)\beta^2+\beta &\hdots& \beta^2+\beta &  \beta\\
(n-2)\beta^2+\beta  &(n-2)\beta^2-\lambda^2+1& (n-3)\beta^2+\beta&\hdots & \beta^2+\beta & \beta\\
(n-3)\beta^2+\beta&   (n-3)\beta^2+\beta&  & & & \beta\\
\vdots&\vdots &   & & & \vdots\\
\beta^2+\beta & \beta^2+\beta&   & &\beta^2-\lambda^2+1& \beta\\
\beta& \beta &\beta&\hdots &\beta&-\lambda^2+1
\end{pmatrix}.
\end{align*}
We rearrange the resulting determinant by subtracting successively the second column from the first, the third from the second, the $n$-th from the $n-1$-th and leave the $n$-th unchanged. This yields
\begin{align*}
&\det{(\tilde{M}_n(\beta)^2-\lambda^2\id)}=\\
&\det\begin{pmatrix}
\beta^2-\beta-\lambda^2+1  &\beta^2 &\beta^2 &\hdots& \beta^2 &  \beta\\
\beta+\lambda^2-1&\beta^2-\beta-\lambda^2+1& \beta^2&\hdots & \beta^2& \beta\\
0&\beta+\lambda^2-1&  & & & \beta\\
\vdots&\vdots &   & & & \vdots\\
0 & 0&   & &\beta^2-\beta-\lambda^2+1& \beta\\
0& 0&0&\hdots &\beta+\lambda^2-1&-\lambda^2+1
\end{pmatrix}.
\end{align*}
Similarly, we subtract the second row from the first, the third from the second, the $n$-th from the $n-1$-th and leave the $n$-th unchanged. We conclude
\begin{align}
&\det{(\tilde{M}_n(\beta)^2-\lambda^2\id)}=\nonumber\\
&\textnormal{\small{$\det\begin{pmatrix}
\beta^2-2\beta-2\lambda^2+2  &\beta+\lambda^2-1 &0 &\hdots& 0 &  0\\
\beta+\lambda^2-1&\beta^2-2\beta-2\lambda^2+2& \beta+\lambda^2-1&0 &\hdots& 0\\
0&\beta+\lambda^2-1&  & & & 0\\
\vdots&0 &   & & & \vdots\\
0 & \vdots&   & &\beta^2-2\beta-2\lambda^2+2&\beta+\lambda^2-1\\
0& 0&0&\hdots &\beta+\lambda^2-1&-\lambda^2+1
\end{pmatrix}$}}=\nonumber\\
&\textnormal{\small{$\det\begin{pmatrix}
\beta^2-2\beta-2\lambda^2+2  &\beta+\lambda^2-1 &0 &\hdots& 0 &  0\\
\beta+\lambda^2-1&\beta^2-2\beta-2\lambda^2+2& & &\hdots& 0\\
0&\beta+\lambda^2-1&  & & & 0\\
\vdots&0 &   & & & \vdots\\
0 & \vdots&   & &\beta^2-2\beta-2\lambda^2+2&\beta+\lambda^2-1\\
0& 0&0&\hdots &\beta+\lambda^2-1&\beta^2-2\beta-2\lambda^2+2
\end{pmatrix}$}}\nonumber\\
&+\textnormal{\small{$\det\begin{pmatrix}
\beta^2-2\beta-2\lambda^2+2  &\beta+\lambda^2-1 &0 &\hdots& 0 &  0\\
\beta+\lambda^2-1&\beta^2-2\beta-2\lambda^2+2& & &\hdots& 0\\
0&\beta+\lambda^2-1&  & & & 0\\
\vdots&0 &   & & & \vdots\\
0 & \vdots&   & &\beta^2-2\beta-2\lambda^2+2&0\\
0& 0&0&\hdots &\beta+\lambda^2-1&\lambda^2-(\beta-1)^2
\end{pmatrix}$}},\label{wutte}
\end{align}
where the last equality is a consequence of the linearity of $\textnormal{det}$ in the last column.
The following is a classical formula for the determinant of an $n\times n$ tri-diagonal Toeplitz matrix~\cite{eger,pascal}
\begin{align}
\det{\begin{pmatrix}
x  &1&0 &\hdots&  0\\
1&x&1&\ddots &\vdots\\
0&1& \ddots&\ddots& 0\\
\vdots &\ddots &\ddots &x&1\\
0& \hdots&0 &1&x
\end{pmatrix}}=\frac{\sin(n+1)\theta}{\sin\theta},\qquad x=2\cos\theta.\label{pascal}
\end{align}
To apply this result we exclude the trivial case $\beta=0$ and note that we can always assume that $\lambda^2\geq1$ such that $\beta+\lambda^2-1>0$ and $\frac{\beta^2}{\beta+\lambda^2-1}\in(0,\beta]$.
Hence, we can divide all columns of both determinants of \eqref{wutte} by $\beta+\lambda^2-1$. We then expand the second determinant along its last column and apply~\eqref{pascal} to both terms resulting from~\eqref{wutte}. We find
\begin{align}
\det{(\tilde{M}_n(\beta)^2-\lambda^2\id)}={(\beta+\lambda^2-1)^n}\left(\frac{\sin(n+1)\theta}{\sin\theta}+\frac{\lambda^2-(\beta-1)^2}{\lambda^2+(\beta-1)}\:\frac{\sin n\theta}{\sin\theta}\right)\label{zerowant}
\end{align}
with
\begin{align*}
2\cos\theta=\frac{\beta^2-2\beta-2\lambda^2+2}{\lambda^2+\beta-1}=\frac{\beta^2}{\lambda^2+\beta-1}-2.
\end{align*}
Solving the latter for $\lambda^2$ gives
\begin{align*}
\lambda^2=\frac{1}{4}\left((\beta-2)^2+\beta^2\tan^2(\theta/2)\right),
\end{align*}
where $\beta\neq0$ implies that $\theta$ is such that the tangent is well defined. This enables us to eliminate $\lambda^2$ from \eqref{zerowant} as
\begin{align*}
\frac{\lambda^2-(\beta-1)^2}{\lambda^2+(\beta-1)}=\frac{1}{\beta}(-\beta+2-2\beta\cos\theta+2\cos\theta).
\end{align*}
It follows that~\eqref{zerowant} is zero if and only if
\begin{align*}
0&=\beta\:\frac{\sin (n+1)\theta}{\sin\theta}+(-\beta+2-2\beta\cos\theta+2\cos\theta)\:\frac{\sin n\theta}{\sin\theta}\\
&=\beta\:\cos n\theta+(2-\beta)\:(1+\cos\theta)\:\frac{\sin n\theta}{\sin\theta},
\end{align*}
which in turn is equivalent to
\begin{align}
\cot n\theta=\frac{\beta-2}{\beta}\cot(\theta/2)\label{cot}.
\end{align}

In total, we are looking for the solution $\theta^*$ of \eqref{cot} such that $\lambda^2$ is maximal i.e. $\cot^2(\theta^*/2)$ is minimal. Since for any $\theta\in[\frac{2n-1}{2n}\pi,\pi)$ we have $\frac{\beta-2}{\beta}\cot(\theta/2)\leq0$ with $\cot(\pi/2)=0$
and $\lim_{\theta\uparrow\pi}\cot{n\theta}\to-\infty$ and $\cot{\frac{2n-1}{2}\pi}=0$, it follows that there is a unique solution $\theta^*\in[\frac{2n-1}{2n}\pi,\pi)$ of Equation~\eqref{cot}. Moreover, by the same fact, $\cot(\pi/2)=0$, this solution maximizes $\lambda^2$ as desired.

Sometimes it is possible to obtain a solution for Equation~\eqref{cot} in closed form. Suppose $\beta=2$,
then $\cot n\theta^*=0$ and $\theta^*=\frac{2n-1}{2n}\pi$. It follows
\begin{align*}
\Norm{M_n(2)}{\infty}^2=\tan^2\left(\frac{2n-1}{4n}\pi\right)=\cot^2(\pi/4n)
\end{align*}
as in \cite{BS}. If $\beta=1$ we have
\begin{align*}
\lambda^2=\frac{1}{4\cos^2(\theta/2)}
\end{align*}
and
\begin{align*}
\sin(2n+1)\theta/2=0
\end{align*}
such that $\theta^*=\frac{2n\pi}{2n+1}$. It follows
\begin{align*}
\Norm{M_n(1)}{\infty}^2=\frac{1}{4\cos^2(\frac{n\pi}{2n+1})}=\frac{1}{4\sin^2(\frac{\pi}{4n+2})}
\end{align*}
as in \cite{eger}. The trivial fact $\Norm{M_n(0)}{\infty}=1$ can be recovered by continuous extension as $\beta\to0$.
\end{proof}

\section{Power-bounded operators}\label{powerbounds}

It is natural to ask if power-boundedness of $A$ is sufficient to obtain estimates on $\Norm{(\zeta-A)^{-1}}{}$ qualitatively similar to the results of Theorem~\ref{SD}, \ref{improve1} and Corollary~\ref{corollary1}. In this section we prove that this is indeed the case and present a new bound on the norm of the resolvent of a power-bounded operator.
\begin{theorem}\label{improve2}
Let $A$ be an $n\times n$ matrix with minimal polynomial $m$ of degree ${\abs{m}}$ and let $\Norm{\cdot}{}$ be an arbitrary matrix norm with
$\sup_{k\geq0}\norm{A^k}{}=C<\infty$. For any $\zeta\in \overline{\mathbb{D}}-\sigma(A)$ it holds that
\begin{align*}
&\Norm{(\zeta-A)^{-1}}{}\\&\leq\frac{2{\abs{m}}C}{\min_{\lambda_i\in\sigma(A)}\abs{1-\bar{\zeta}\lambda_i}^{1/2}(2{\abs{m}}-2{\abs{m}}\abs{\zeta}^2+\abs{\zeta}^2\min_{\lambda_i\in\sigma(A)}\abs{1-\bar{\zeta}\lambda_i})^{1/2}}\left(\frac{4e}{\abs{B(\zeta)}^2}-1\right)^{1/2},
\end{align*}
where $B(\zeta)=\prod_{i=1}^{\abs{m}}\frac{\zeta-\lambda_i}{1-\bar{\lambda}_i\zeta}$ is the Blaschke product associated with $m$.
For $\abs{\zeta}>1$, we have the obvious estimate $\Norm{(\zeta-A)^{-1}}{}\leq\frac{C}{\abs{\zeta}-1}$.
\end{theorem}
%
Theorem~\ref{improve2} is the analogue of Corollary~\ref{corollary1} for power-bounded operators. Spectral bounds on the norm of the resolvent of a power-bounded operator are well studied in the literature. Theorem 6.4 of \cite{BS} treats the same problem in the special case that $A$ is power-bounded with respect to operator norm $\Norm{\cdot}{\infty}$. In \cite[Theorem 3.24]{AC} the behavior of $\Norm{(\zeta-A)^{-1}}{}$ is studied for $\abs{\zeta}<1$ and in \cite{Rachid1} an upper bound is derived for $\abs{\zeta}\geq1$. Theorem~\ref{improve2} unifies the mentioned results and yields a quantitatively better bound in each case.
To compare suppose that $\abs{\zeta}<1$ and note that in this case
\begin{align*}
1-\abs{\zeta}^2+\frac{1}{2\abs{m}}\abs{\zeta}^2\min_{\lambda_i\in\sigma(A)}\abs{1-\bar{\zeta}\lambda_i}\geq(1-\abs{\zeta})^2
\end{align*}
and of course
$\min_{\lambda_i\in\sigma(A)}\abs{1-\bar{\zeta}\lambda_i}\geq1-\abs{\zeta}$. 
Hence, it follows
\begin{align*}
\Norm{(\zeta-A)^{-1}}{}\leq\frac{\sqrt{8e\abs{m}}C}{(1-\abs{\zeta})^{3/2}}\:\frac{1}{\abs{B(\zeta)}},
\end{align*}
which is qualitatively the same as Theorem 3.24 in \cite{AC} but has a better numerical prefactor. If we choose $\abs{\zeta}=1$ it follows $\abs{B(\zeta)}=1$ and therefore
\begin{align}
\Norm{(\zeta-A)^{-1}}{}\leq\frac{\sqrt{16e-4}\:\abs{m}C}{\min_{\lambda_i\in\sigma(A)}\abs{\zeta-\lambda_i}}.\label{rach}
\end{align}
This bound improves on the result in \cite{Rachid1} (which in turn is stronger than \cite[Theorem 6.4]{BS}) as the new bound only grows linearly with $\abs{m}$ as opposed to $\abs{m}^{3/2}$ in \cite{Rachid1}. That for power-bounded $A\in\cM_n$ the correct asymptotic growth order for an upper bound is $\cO(n)$ was already suspected in \cite{BS} and \cite{Rachid2}. The bound obtained almost reaches the optimal estimate of Theorem~\ref{SD} for Hilbert-space contractions. In the latter case we have that $\cot(\frac{\pi}{4n})/n\leq\frac{4}{\pi}$, while the prefactor of \eqref{rach} is $\sqrt{16e-4}\approx6.28$. However, as is clear from the derivation, Inequality~\eqref{rach} is not optimal. We will use Inequality~\eqref{rach} to study the sensitivity of a classical or quantum Markov chain to perturbations in Section~\ref{Chain}.

To prove Theorem~\ref{improve2} we take a similar approach as to Theorem~\ref{improve1}. We note that power-bounded operators admit a Wiener algebra functional calculus. Thus an application of Lemma~\ref{ratmod} reveals that
\begin{align}
\Norm{(\zeta-A)^{-1}}{}\leq C\inf\{\Norm{g}{W}\:|\:g\in W,\ g(\lambda_i)=\frac{1}{\zeta-\lambda_i}\}\label{here}.
\end{align}

The strategy of our proof will be to consider one specific representative function $g$ in \eqref{here} and to bound its norm. To achieve this we employ the following method. Instead of considering $g$ directly we choose a ``smoothing parameter\rq\rq{} $r$ and pass to a ``stretched\rq\rq{} interpolation function.\\
Given any function $f\in H_2$ and $r\in(0,1)$, we write $f_r(z):=f(rz)=\sum_{k\geq0}\hat{f}(k)r^kz^k$ and observe that by the Cauchy-Schwarz inequality and the Plancherel identity
\begin{align}
\Norm{f_r}{W}\leq\sqrt{\sum_{k\geq0}\abs{\hat{f}(k)}^2}\sqrt{\frac{1}{1-r^2}}=\Norm{f}{H_2}\sqrt{\frac{1}{1-r^2}}.\label{CS}
\end{align}
This idea was used to obtain bounds on the inverse and resolvent of a power-bounded operator in \cite{AC} and to study spectral convergence bounds for bounded semigroups in \cite{wir}.\\ 
We use the Blaschke products $B(z)=\prod_i\frac{z-\lambda_i}{1-\bar{\lambda}_i z}$ and $\tilde{B}(z)=\prod_i\frac{z-r\lambda_i}{1-r\bar{\lambda}_i z}$, where in the latter product the spectrum is stretched by a factor of $r$. (The products are taken over all prime factors of $m$, but to avoid cumbersome notation we do not write this explicitly.) Consider now the function $g$ with
\begin{align*}
g(z)=\sum_k \left(\frac{1}{\zeta-\lambda_k}\frac{\prod_j(1-\bar{\lambda}_j\lambda_k)}{\prod_{j\neq k}(\lambda_k-\lambda_j)}\right)\frac{B(z)}{z-\lambda_k}.
\end{align*}
Note that $g$ is analytic in the unit disc and $g(\lambda_i)=\frac{1}{\zeta-\lambda_i}$ for all $\lambda_i\in\sigma(A)$. In order to use the estimate~\eqref{CS} we perform the aforementioned smoothing. We define the modified function $\tilde{g}$ by
\begin{align*}
\tilde{g}(z)=\sum_k\left( \frac{1}{\zeta-\lambda_k}\frac{\prod_j(1-r^2\bar{\lambda}_j\lambda_k)}{\prod_{j\neq k}(r\lambda_k-r\lambda_j)}\right)\frac{\tilde{B}(z)}{z-r\lambda_k}
\end{align*}
and observe that $\tilde{g}_r$ enjoys the same basic properties as $g$ i.e.~$\tilde{g}_r$ is analytic in $\mathbb{D}$ and $\tilde{g}_r(\lambda_i)=\frac{1}{\zeta-\lambda_i}$ for any $\lambda_i\in\sigma(A)$.
Thus, by Inequality~\eqref{here}, we have that $\Norm{(\zeta-A)^{-1}}{}\leq C\Norm{\tilde{g}_r}{W}$ and it follows from Inequality~\eqref{CS} that
\begin{align*}
\Norm{\tilde{g}_r}{W}\leq\sqrt{\frac{1}{1-r^2}}\Norm{\tilde{g}}{H_2}.
\end{align*}
It turns out that one can directly compute $\Norm{\tilde{g}}{H_2}$. The computation relies on two combinatorial observations similar to Lemma~\ref{combi1}, which we shall prove before we proceed with our discussion of $\Norm{\tilde{g}}{H_2}$.
\begin{lemma}\label{combi2} Let $\abs{m}\in\mathbb{N}-\{0\}$ and $\{\lambda_i\}_{i=1,...,\abs{m}}\subset\mathbb{D}$. Furthermore, let $\zeta\in\mathbb{C}-\{\lambda_i\}_{i=1,...,\abs{m}}$ and $r\in(0,1)$. It follows that
\begin{enumerate}\item
\begin{align*}
\sum_{i=1}^{\abs{m}}\frac{1}{\zeta-\lambda_i}\frac{\prod_{j:j\neq l}(1-r^2\bar{\lambda}_j\lambda_i)}{\prod_{j:j\neq i}(r\lambda_i-r\lambda_j)}=\frac{r}{1-r^2\bar{\lambda}_l\zeta}\:\prod_{i=1}^{\abs{m}}\frac{1-r^2\bar{\lambda}_i\zeta}{r\zeta-r\lambda_i},
\end{align*}
\item
\begin{align*}
\sum_{i=1}^{\abs{m}}\frac{1}{\zeta-\lambda_i}\:\frac{1}{1-r^2\bar{\zeta}\lambda_i}\frac{\prod_j(1-r^2\bar{\lambda}_j\lambda_i)}{\prod_{j:j\neq i}(r\lambda_i-r\lambda_j)}=\frac{r}{1-r^2\abs{\zeta}^2}\left(\prod_{i=1}^{\abs{m}}\frac{1-r^2\bar{\lambda}_i\zeta}{r\zeta-r\lambda_i}-\prod_{i=1}^{\abs{m}}\frac{r\bar{\zeta}-r\bar{\lambda}_i}{1-r^2\lambda_i\bar{\zeta}}\right),
\end{align*}
\item
\begin{align*}
\Norm{\tilde{g}}{H_2}^2=\frac{r^2}{1-r^2\abs{\zeta}^2}\left(\prod_{i=1}^{\abs{m}}\Abs{\frac{1-r^2\bar{\lambda}_i\zeta}{r\zeta-r\lambda_i}}^2-1\right).
\end{align*}
\end{enumerate}
\end{lemma}
Our proof is based on the Residue theorem. (It is also possible to prove the lemma with the second technique outlined in the proof of Lemma~\ref{combi1}. However, the approach via the Residue theorem is more convenient for the second assertion.)
\begin{proof}
For the first assertion set $t_1(z):=\frac{rz}{r\zeta-z}\frac{1}{1-r\bar{\lambda}_lz}$ and suppose for now that $r\abs{\zeta}<1$. We have that
\begin{align*}
\braket{t_1}{\tilde{B}}&=\int_0^{2\pi}\frac{rz}{r\zeta-z}\frac{1}{1-r\bar{\lambda}_lz}\:\prod_i\frac{1-r\bar{\lambda}_i z}{z-r\lambda_i}\Big|_{z=e^{i\phi}}\frac{\d\phi}{2\pi}\\
&=\frac{1}{2\pi i}\int_{\partial\mathbb{D}}\frac{r}{r\zeta-z}\frac{1}{1-r\bar{\lambda}_lz}\:\prod_i\frac{1-r\bar{\lambda}_i z}{z-r\lambda_i}\d z\\
&=\sum_i\frac{1}{\zeta-\lambda_i}\frac{\prod_{j:j\neq l}(1-r^2\bar{\lambda}_j\lambda_i)}{\prod_{j:j\neq i}(r\lambda_i-r\lambda_j)}-\frac{r}{1-r^2\bar{\lambda}_l\zeta}\prod_i\frac{1-r^2\bar{\lambda}_i\zeta}{r\zeta-r\lambda_i}
\end{align*}
and that
\begin{align*}
\braket{\tilde{B}}{t_1}&=\int_0^{2\pi}\prod_i\frac{z-r\lambda_i}{1-r\bar{\lambda}_i z}\frac{r}{r\bar{\zeta}z-1}\frac{z}{z-r\lambda_l}\:\Big|_{z=e^{i\phi}}\frac{\d\phi}{2\pi}\\
&=\frac{1}{2\pi i}\int_{\partial\mathbb{D}}\prod_i\frac{z-r\lambda_i}{1-r\bar{\lambda}_i z}\frac{r}{r\bar{\zeta}z-1}\frac{1}{z-r\lambda_l}\:\d z=0.
\end{align*}
Hence, for $r\abs{\zeta}<1$
\begin{align*}
\sum_i\frac{1}{\zeta-\lambda_i}\frac{\prod_{j:j\neq l}(1-r^2\bar{\lambda}_j\lambda_i)}{\prod_{j:j\neq i}(r\lambda_i-r\lambda_j)}=\frac{r}{1-r^2\bar{\lambda}_l\zeta}\prod_i\frac{1-r^2\bar{\lambda}_i\zeta}{r\zeta-r\lambda_i}
\end{align*}
as claimed. An identical computation reveals that the above remains correct if $r\abs{\zeta}>1$ and the case $r\abs{\zeta}=1$ follows by continuity. For the second assertion suppose again that $r\abs{\zeta}<1$ and set $t_2(z):=\frac{rz}{r\zeta-z}\frac{1}{1-r\bar{\zeta}z}$ and compute
\begin{align*}
\braket{t_2}{\tilde{B}}&=\int_0^{2\pi}\frac{rz}{r\zeta-z}\frac{1}{1-r\bar{\zeta}z}\:\prod_i\frac{1-r\bar{\lambda}_i z}{z-r\lambda_i}\Big|_{z=e^{i\phi}}\frac{\d\phi}{2\pi}\\
&=\frac{1}{2\pi i}\int_{\partial\mathbb{D}}\frac{r}{r\zeta-z}\frac{1}{1-r\bar{\zeta}z}\:\prod_i\frac{1-r\bar{\lambda}_i z}{z-r\lambda_i}\:\d z\\
&=\sum_i\frac{1}{\zeta-\lambda_i}\frac{1}{1-r^2\bar{\zeta}\lambda_i}\frac{\prod_j(1-r^2\bar{\lambda}_j\lambda_i)}{\prod_{j:j\neq i}(r\lambda_i-r\lambda_j)}-\frac{r}{1-r^2\abs{\zeta}^2}\:\prod_i\frac{1-r^2\bar{\lambda}_i\zeta}{r\zeta-r\lambda_i}.
\end{align*}
Similarly,
\begin{align*}
\braket{\tilde{B}}{t_2}=\frac{1}{2\pi i}\int_{\partial\mathbb{D}}\prod_i\frac{z-r\lambda_i}{1-r\bar{\lambda}_i z}\frac{r}{r\bar{\zeta}z-1}\frac{1}{z-r\zeta}\:\d z=\frac{r}{r^2\abs{\zeta}^2-1}\prod_i\frac{r\zeta-r\lambda_i}{1-r^2\bar{\lambda}_i \zeta}.
\end{align*}
It follows that
\begin{align*}
\sum_i\frac{1}{\zeta-\lambda_i}\frac{1}{1-r^2\bar{\zeta}\lambda_i}\frac{\prod_j(1-r^2\bar{\lambda}_j\lambda_i)}{\prod_{j:j\neq i}(r\lambda_i-r\lambda_j)}=\frac{r}{1-r^2\abs{\zeta}^2}\left(\prod_i\frac{1-r^2\bar{\lambda}_i\zeta}{r\zeta-r\lambda_i}-\prod_i\frac{r\bar{\zeta}-r\bar{\lambda}_i}{1-r^2\lambda_i\bar{\zeta}}\right).
\end{align*}
The same computations prove the validity of this statement for $r\abs{\zeta}>1$. One can make sense of the formula in case that $r\abs{\zeta}=1$ by continuous extension. Using these observations one can compute
\begin{align}
\Norm{\tilde{g}}{H_2}^2&=\int_0^{2\pi}\tilde{g}(z)\overline{\tilde{g}(z)}\Big|_{z=e^{i\phi}}\frac{\d\phi}{2\pi}\nonumber\\
&=\frac{1}{2\pi i}\sum_{k,l}\frac{1}{\zeta-\lambda_k}\frac{\prod_\mu(1-r^2\bar{\lambda}_\mu\lambda_k)}{\prod_{\mu\neq k}(r\lambda_k-r\lambda_\mu)}\overline{\frac{1}{\zeta-\lambda_l}\frac{\prod_\nu(1-r^2\bar{\lambda}_\nu\lambda_l)}{\prod_{\nu\neq l}(r\lambda_l-r\lambda_\nu)}}\int_{\partial\mathbb{D}}\frac{1}{z-r\lambda_k}\frac{1}{1-r\bar{\lambda}_lz}\d z\nonumber\\
&=\sum_l\left(\overline{\frac{1}{\zeta-\lambda_l}\frac{\prod_\nu(1-r^2\bar{\lambda}_\nu\lambda_l)}{\prod_{\nu\neq l}(r\lambda_l-r\lambda_\nu)}}\sum_k\left(\frac{1}{\zeta-\lambda_k}\frac{1}{1-r^2\bar{\lambda}_l\lambda_k}\frac{\prod_\mu(1-r^2\bar{\lambda}_\mu\lambda_k)}{\prod_{\mu\neq k}(r\lambda_k-r\lambda_\mu)}\right)\right)\nonumber\\
&=\prod_{i}\frac{1-r^2\bar{\lambda}_i\zeta}{r\zeta-r\lambda_i}\overline{\left(\sum_l\frac{1}{\zeta-\lambda_l}\frac{r}{1-r^2\lambda_l\bar{\zeta}}\frac{\prod_\nu(1-r^2\bar{\lambda}_\nu\lambda_l)}{\prod_{\nu\neq l}(r\lambda_l-r\lambda_\nu)}\right)}\label{ass1}\\
&=\frac{r^2}{1-r^2\abs{\zeta}^2}\prod_{i}\frac{1-r^2\bar{\lambda}_i\zeta}{r\zeta-r\lambda_i}\overline{\left(\prod_{i}\frac{1-r^2\bar{\lambda}_i\zeta}{r\zeta-r\lambda_i}-\prod_{i}\frac{r\bar{\zeta}-r\bar{\lambda}_i}{1-r^2\lambda_i\bar{\zeta}}\right)}\label{ass2}\\
&=\frac{r^2}{1-r^2\abs{\zeta}^2}\left(\prod_{i}\Abs{\frac{1-r^2\bar{\lambda}_i\zeta}{r\zeta-r\lambda_i}}^2-1\right)\nonumber,
\end{align}
where we used the first assertion of the lemma for \eqref{ass1} and the second assertion for \eqref{ass2}. Note that for all $\zeta\in\mathbb{C}-\sigma(A)$ and $r\in(0,1)$ the final quantity is real and positive.
\end{proof}
With this preparatory work done a proof of Theorem~\ref{improve2} is simple.
\begin{proof}[Proof of Theorem~\ref{improve2}] We assume that $\sigma(A)\subset{\mathbb{D}}$.
From Equations~\eqref{here},~\eqref{CS} and Lemma~\ref{combi2}
we have that for any $\zeta\in\mathbb{C}-\sigma(A)$
\begin{align}\label{coreth}
\Norm{(\zeta-A)^{-1}}{}\leq C\sqrt{\frac{1}{1-r^2}}\:\Norm{\tilde{g}}{H_2}= C\sqrt{\frac{1}{1-r^2}}\sqrt{\frac{r^2}{1-r^2\abs{\zeta}^2}\left(\prod_{i=1}^{\abs{m}}\Abs{\frac{1-r^2\bar{\lambda}_i\zeta}{r\zeta-r\lambda_i}}^2-1\right)}.
\end{align}
Clearly,
\begin{align*}
\prod_{i=1}^{\abs{m}}\Abs{\frac{1-r^2\bar{\lambda}_i\zeta}{r\zeta-r\lambda_i}}^2=\frac{1}{r^{2\abs{m}}}\:\frac{1}{\abs{B(\zeta)}^2}\prod_{i=1}^{\abs{m}}\Abs{1+\bar{\lambda}_i\zeta\:\frac{1-r^2}{1-\bar{\lambda}_i\zeta}}^2.
\end{align*}
To obtain an upper bound we assume that $\zeta\in\overline{\mathbb{D}}-\sigma(A)$ and choose $r\in(0,1)$ such that
\begin{align*}
1-r^2=\frac{\min_i\abs{1-\bar{\zeta}\lambda_i}}{2{\abs{m}}}.
\end{align*}
It follows that
\begin{align*}
\prod_{i=1}^{\abs{m}}\Abs{1+\bar{\lambda}_i\zeta\:\frac{1-r^2}{1-\bar{\lambda}_i\zeta}}^2\leq\left(1+\frac{1}{2{\abs{m}}}\right)^{2{\abs{m}}}\leq e
\end{align*}
and that (for ${\abs{m}}\geq2$)
\begin{align*}
r^{2{\abs{m}}}=\left(1-\frac{\min_i\abs{1-\bar{\zeta}\lambda_i}}{2\abs{m}}\right)^{\abs{m}}\geq(1-1/2)^2=1/4.
\end{align*}
We conclude that
\begin{align*}
\Norm{(\zeta-A)^{-1}}{}\leq C\left(\frac{2{\abs{m}}}{\min_i\abs{1-\bar{\zeta}\lambda_i}}\right)^{1/2}\left(\frac{1}{1-\abs{\zeta}^2\left(1-\frac{\min_i\abs{1-\bar{\zeta}\lambda_i}}{2{\abs{m}}}\right)}\right)^{1/2}\left(\frac{4e}{\abs{B(\zeta)}^2}-1\right)^{1/2},
\end{align*}
which is claimed in the theorem. As always the general case $\sigma(A)\subset\overline{\mathbb{D}}$ follows by continuous extension. Finally, we note that for $\abs{\zeta}>1$ one can choose $r=\sqrt{\frac{1}{\abs{\zeta}}}$ in \eqref{coreth} and recover the obvious estimate
\begin{align*}
\Norm{(\zeta-A)^{-1}}{}\leq C\frac{1}{\abs{\zeta}-1}\left(1-\prod_{i=1}^{\abs{m}}\Abs{\frac{\abs{\zeta}-\bar{\lambda}_i\zeta}{{\sqrt{\abs{\zeta}}}\zeta-{\sqrt{\abs{\zeta}}}\lambda_i}}^2\right)^{1/2}\leq C\frac{1}{\abs{\zeta}-1}.
\end{align*}\end{proof}
\section{Stability of Markov chains}\label{Chain}

If $T$ is a classical stochastic matrix or a quantum channel (a trace-preserving and completely positive map, see \cite{NielsenChuang}) the sequence $\{T^n\}_{n\geq0}$ can be regarded as a finite and homogenous classical or quantum Markov chain with transition map $T$. In this section we apply Theorem~\ref{improve2} to study the stability of the stationary states of a Markov chain to perturbations in the transition map. A core observation is that the transition matrix of the Markov chain is power-bounded with respect to the 1-to-1 norm and constant $1$, i.e. the Markov chain constitutes a bounded semigroup, see \cite{wir}.  A similar approach based on this observation was taken in \cite{szrewo} to prove spectral convergence estimates for classical and quantum Markov chains. We begin by recalling the basic framework of sensitivity analysis of Markov chains. A detailed introduction, however, is beyond the scope of this article. We refer to \cite{Overview} and the references therein for an overview of the existing perturbation bounds for classical Markov chains and to \cite{wir} for an introduction to the stability theory of quantum Markov chains.

Let $T,\ \tilde{T}$ denote two classical stochastic matrices or two quantum channels. The inequality
\begin{align*}
\Norm{\rho-\tilde{\rho}}{}\leq\kappa\Norm{T-\tilde{T}}{}
\end{align*}
relates the distance between two stationary states $\rho$ and $\tilde{\rho}$ arising from $T$ and $\tilde{T}$, $\rho=T(\rho),\ \tilde{\rho}=\tilde{T}(\tilde{\rho})$, to the distance between $T$ and $\tilde{T}$. Commonly $T$ is considered to be the transition matrix of the Markov chain of interest while $\tilde{T}$ is a small perturbation thereof. The \emph{condition number} $\kappa$ measures the relative sensitivity of the stationary states to perturbations. If $T$ has a unique stationary state the above inequality quantifies the stability of the asymptotic behavior of $\{T^n\}_{n\geq0}$ with respect to perturbations in the transition matrix. Elementary linear algebra shows that if $T$ has a unique stationary state one can choose (see \cite{Schweitzer,Sen1,Sen3}) the condition number
\begin{align*}
\kappa_{cl}=\sup_{\delta\in\mathbb{R}^n\atop(1,...,1)\cdot\delta=0}\frac{\Norm{Z(\delta)}{1}}{\Norm{\delta}{1}},\quad Z:=(\id-T+T^\infty)^{-1}
\end{align*}
in the classical case and similarly (see~\cite{wir})
\begin{align*}
\kappa_{qu}=\sup_{\sigma=\sigma^\dagger\in \cM_n\atop{\tr{(\sigma)}=0}}\frac{\Norm{Z(\sigma)}{1}}{\Norm{\sigma}{1}},\quad Z:=(\id-T+T^\infty)^{-1}
\end{align*}
in the quantum setup. Here, $T^\infty$ denotes the projection onto the stationary state of $T$ and $\Norm{\cdot}{1}$ denotes the absolute entry sum in the classical and the Schatten 1-norm in the quantum case. In either case the spectral properties of $T$ and $T^\infty$ guarantee that the map $Z$ exists.

If the transition matrix has a unique stationary state and a subdominant eigenvalue of this matrix is close to $1$ it is clear that the chain is ill conditioned in the sense that the stationary state is sensitive to perturbations in the transition map. It is a well-studied question (see \cite{meyerrr,Sen3,Mit1,Mit2,wir}) whether the reverse conclusion also holds, i.e.~whether the closeness of the sub-dominant eigenvalues of $T$ to $1$ provides complete information about the sensitivity of $\{T^n\}_{n\geq0}$. It was established that this is indeed the case by deriving spectral lower and upper bounds for certain choices of $\kappa$. In particular, as shown in \cite{Sen3} it holds that
\begin{align*}
\frac{1}{\min_{\lambda_i\in\sigma(T-T^\infty)}\abs{1-\lambda_i}}\leq\kappa_{cl}\leq\frac{n}{\min_{\lambda_i\in\sigma(T-T^\infty)}\abs{1-\lambda_i}}.
\end{align*}
A similar quantum bound occurs in \cite{wir}.

The techniques developed in this article yield a direct approach to spectral stability estimates in both the classical and quantum case. The core observation is that if $T$ is a stochastic matrix or a quantum channel the map $T-T^\infty$ is power bounded with (see \cite{szrewo} Lemma III.1)
\begin{align*}		
\Norm{(T-T^\infty)^n}{1\to1}=\Norm{T^n-(T^\infty)^n}{1\to1}\leq\Norm{T^n}{1\to1}+\Norm{(T^\infty)^n}{1\to1}\leq2,
\end{align*}
where $\Norm{A}{1\to1}=\sup_{v}\frac{\Norm{Av}{1}}{\Norm{v}{1}}$.
With an application of Inequality~\eqref{rach} we conclude that
\begin{align*}
\kappa_{cl}\leq\Norm{Z}{1\to1}\leq\frac{2\sqrt{16e-4}n}{\min_{\lambda_i\in\sigma(T-T^\infty)}\abs{1-\lambda_i}},
\end{align*}
which is qualitatively the same as the estimate in \cite{Sen3}  but has a worse numerical prefactor ($2\sqrt{16e-4}$ instead of $1$). However, the bound in \cite{Sen3} uses the additional properties a classical stochastic matrix has as well as the fact that the supremum in the definition of $\kappa_{cl}$ is taken over vectors with $0$ column sum.
Our bound proves that in this case power-boundedness alone is sufficient and the additional assumptions on $T$ and $\kappa_{cl}$ are basically superfluous. Other spectral stability estimates for classical Markov chains such as \cite{meyerrr} are weaker than \eqref{rach}. In the quantum context we can use Inequality~\eqref{rach} to improve on the spectral stability estimates of \cite{wir}.
\begin{theorem}
Let $T$ be a trace-preserving, positive linear map on $\cM_n$ and {$\Lambda:=\sigma(T)-\{1\}$} the set of its non-unit eigenvalues. Then 
$$\frac{1}{\min_{\lambda_i\in\Lambda}\abs{1-\lambda_i}}\leq\kappa_{qu}\leq\frac{2\sqrt{16e-4}n^2}{\min_{\lambda_i\in\Lambda}\abs{1-\lambda_i}}.$$
\end{theorem}
The proof of the theorem is identical as in \cite{wir} up to an application of \eqref{rach} instead of the theorem in \cite{Rachid1}.
\begin{acknowledgements}
OS acknowledges financial support by the Elite Network of Bavaria (ENB) project QCCC and the
CHIST-ERA/BMBF project CQC. OS is thankful to Michael M. Wolf for creating conditions that made this work possible and to Alexander M\"{u}ller-Hermes for proofreading the manuscript and for pointing out the simpler proof for Lemma~\ref{combi1}. OS is equally thankful to E.\:B.\:Davies for valuable comments on a previous version of the manuscript.
\end{acknowledgements}

%
%



\end{document}